\documentclass[11pt, a4paper]{article}

\usepackage[utf8]{inputenc}
\usepackage[english]{babel}

\usepackage{amsmath, amsthm, amsfonts, enumerate, graphicx, tikz, pgf, amssymb, subfloat, subfig}

\usetikzlibrary{arrows}

\usepackage[pdftex,unicode]{hyperref}

\theoremstyle{plain}
\newtheorem{thm}{Theorem}[section]
\newtheorem{lemma}[thm]{Lemma}
\newtheorem{fact}[thm]{Fact}
\newtheorem{obs}[thm]{Observation}
\newtheorem{cor}[thm]{Corollary}
\newtheorem{prop}[thm]{Proposition}

\newtheorem*{question}{Question}
\theoremstyle{definition}

\theoremstyle{remark}

\def\R{\mathbb{R}}
\def\A{\mathcal{A}}
\def\B{\mathcal{B}}
\def\C{\mathcal{C}}
\def\D{\mathcal{D}}
\def\F{\mathcal{F}}
\def\Ho{\mathcal{H}}
\def\I{\mathcal{I}}
\def\La{\mathcal{L}}

\def\S{\mathcal{S}}
\def\V{\mathcal{V}}

\def\Av{\textnormal{Av}}
\def\gr{\textnormal{gr}}

\def\c{\circ}

\title{Composability of Permutation Classes}
\author{Mark Karpilovskij\thanks{This work was supported by the Neuron Fund for Support of Science} \\ karpilo@iuuk.mff.cuni.cz
\\ Computer Science Institute of Charles University \\ Malostranské nám. 25, 118 00 Prague 1, Czech Republic}
\date{}

\begin{document}
\maketitle

\begin{abstract}
We define the operation of composing two hereditary classes of permutations
 using the standard composition of permutations as functions and we explore
 properties and structure of permutation classes considering this operation.
 We mostly concern ourselves with the problem of whether permutation classes
 can be composed from their proper subclasses. We provide examples of classes which can be composed from two proper subclasses,
 classes which can be composed from three but not from two proper subclasses and classes which cannot be composed
from any finite number of proper subclasses.
\end{abstract}
\section{Introduction}

Permutations of numbers or other finite sets are a~very deeply and frequently studied
combinatorial and algebraic object. There are two main structures on permutations investigated
in modern mathematics: groups, closed under the composition operator, and hereditary pattern-avoiding classes,
closed under the relation of containment. This paper is one of several texts
exploring the relation between the two notions by applying the composition operator to
permutation classes. That is, given two classes $\A$ and $\B$, we denote by $\A \c \B$ the
class of all permutations which can be written as a~composition of a~permutation from $\A$
and a~permutation from $\B$.

The oldest results combining permutation classes and groups that we know of are due
to Atkinson and Beals \cite{AtkinsonBeals01}, who consider the permutation classes
whose permutations of length $n$ form a~subgroup of $S_n$ for every $n$ and completely
characterise the types of groups which may occur this way. These results were recently
refined and extended by Lehtonen and Pöschel \cite{Lehtonen16, LehtonenPoschel16}.
In an earlier version of their paper, Atkinson and Beals \cite{AtkinsonBeals99} also deal with composing permutation classes,
showing that compositions of many pairs of finitely based classes are again finitely based.

Some permutation classes characterise permutations which can be sorted
by some sorting machine such as a~stack. In this view, a~composition of two permutation
classes can characterise permutations sortable by two corresponding sorting machines connected
serially. For example, Atkinson and Stitt \cite[Section 6.4]{AtkinsonStitt02} introduce 
the pop-stack, a~sorting machine which sorts precisely the layered permutations (see Section \ref{ch4} for a~definition),
and consider the class of permutations which can be sorted by two pop-stacks in series, i.e.
which can be written as a~composition of two layered permutations.
Using their more general results they calculate its generating function and enumerate its basis.

Albert et al. \cite{AlbertEtAl07} give more enumerative results on compositions of classes
in terms of sorting machines. 

In the present paper, we study a~different question connected to compositions of classes; 
namely whether a~permutation of a~given class $\C$ can always be written as a~composition
of two or more permutations from its subclasses, i.e. whether $\C \subseteq \C_1 \c \C_2 \c \cdots \c \C_k$
for some $\C_1, \ldots, \C_k \subsetneq \C$. If this is true, we say that the class $\C$ is \emph{composable}
and we refer to this property of $\C$ as \emph{composability}.

The paper is organised as follows. In Section \ref{ch1} we supply all the necessary
definitions and facts about permutation classes. In Section \ref{ch2}
we introduce composability and give some basic results. In Section \ref{ch3}
we explore composability of the class $\Av(k\cdots21)$. In Section \ref{ch4} we explore
composability of various classes of layered patterns. Finally in Section \ref{ch5}
we give several additional miscellaneous results.

\section{Preliminaries} \label{ch1}

For a~positive integer $n$ we let $[n]$ denote the set $\{1, 2, \ldots, n\}$. A~\emph{permutation of order}
$n$ is a~bijective function $\pi \colon [n] \longrightarrow [n]$. We
denote the order of a~permutation $\pi$ by $|\pi|$.
We may also interpret a~permutation $\pi$ as a~sequence $\pi(1), \pi(2), \ldots, \pi(n)$
of distinct elements of $[n]$, or as a~diagram in an $n \times n$ square in the plane, 
consisting of points $\{(i,\pi(i)); 1 \leq i \leq n\}$. For $n \geq 0$
let $\S_n$ denote the set of all permutations of order $n$.

If $\pi$ and $\sigma$ are two permutations of order $n$ we define their \emph{composition} $\pi \c \sigma$ 
as $(\pi \c \sigma)(i) = \pi(\sigma(i))$ for every $i \in [n]$.

We define two more permutation operators. The \emph{sum} $\pi \oplus \sigma$
of permutations $\pi \in \S_k$ and $\sigma \in \S_l$ is the permutation 
$\pi(1), \pi(2), \ldots, \pi(k), \sigma(1) + k, \sigma(2) + k, \ldots, \sigma(l) + k.$
The \emph{skew sum} $\pi \ominus \sigma$ is the permutation
$\pi(1) + l, \pi(2) + l, \ldots, \pi(k) + l, \sigma(1), \sigma(2), \ldots, \sigma(l).$
For example, $3127645 = 312 \oplus 4312$ and $6547123 = 3214 \ominus 123$ (see Figure \ref{fig_sum}).

\begin{figure}[h] 
  \centering
  \subfloat[][$3127645 = 312 \oplus 4312$] {
    \begin{tikzpicture}[line cap=round,line join=round,>=triangle 45,x=0.4cm,y=0.4cm]
    \clip(0.42,0.52) rectangle (15.46,15.45);
    \draw (1,1)-- (15,1);
    \draw (1,1)-- (1,15);
    \draw (1,15)-- (15,15);
    \draw (15,15)-- (15,1);
    \draw (19,1)-- (33,1);
    \draw (33,1)-- (33,15);
    \draw (33,15)-- (19,15);
    \draw (19,15)-- (19,1);
    \draw (1,7)-- (15,7);
    \draw (19,7)-- (33,7);
    \draw (7,1)-- (7,15);
    \draw (27,15)-- (27,1);
    \begin{scriptsize}
    \fill [color=black] (2,6) circle (2pt);
    \fill [color=black] (4,2) circle (2pt);
    \fill [color=black] (6,4) circle (2pt);
    \fill [color=black] (8,14) circle (2pt);
    \fill [color=black] (10,12) circle (2pt);
    \fill [color=black] (12,8) circle (2pt);
    \fill [color=black] (14,10) circle (2pt);
    \fill [color=black] (20,12) circle (2pt);
    \fill [color=black] (22,10) circle (2pt);
    \fill [color=black] (24,8) circle (2pt);
    \fill [color=black] (26,14) circle (2pt);
    \fill [color=black] (28,2) circle (2pt);
    \fill [color=black] (30,4) circle (2pt);
    \fill [color=black] (32,6) circle (2pt);
    \end{scriptsize}
    \end{tikzpicture}
  }
  \subfloat[][$6547123 = 3214 \ominus 123$] {
    \begin{tikzpicture}[line cap=round,line join=round,>=triangle 45,x=0.4cm,y=0.4cm]
    \clip(18.45,0.52) rectangle (33.56,15.45);
    \draw (1,1)-- (15,1);
    \draw (1,1)-- (1,15);
    \draw (1,15)-- (15,15);
    \draw (15,15)-- (15,1);
    \draw (19,1)-- (33,1);
    \draw (33,1)-- (33,15);
    \draw (33,15)-- (19,15);
    \draw (19,15)-- (19,1);
    \draw (1,7)-- (15,7);
    \draw (19,7)-- (33,7);
    \draw (7,1)-- (7,15);
    \draw (27,15)-- (27,1);
    \begin{scriptsize}
    \fill [color=black] (2,6) circle (2pt);
    \fill [color=black] (4,2) circle (2pt);
    \fill [color=black] (6,4) circle (2pt);
    \fill [color=black] (8,14) circle (2pt);
    \fill [color=black] (10,12) circle (2pt);
    \fill [color=black] (12,8) circle (2pt);
    \fill [color=black] (14,10) circle (2pt);
    \fill [color=black] (20,12) circle (2pt);
    \fill [color=black] (22,10) circle (2pt);
    \fill [color=black] (24,8) circle (2pt);
    \fill [color=black] (26,14) circle (2pt);
    \fill [color=black] (28,2) circle (2pt);
    \fill [color=black] (30,4) circle (2pt);
    \fill [color=black] (32,6) circle (2pt);
    \end{scriptsize}
    \end{tikzpicture}
  }
  \caption{An example of sums and skew sums}
  \label{fig_sum}
\end{figure}
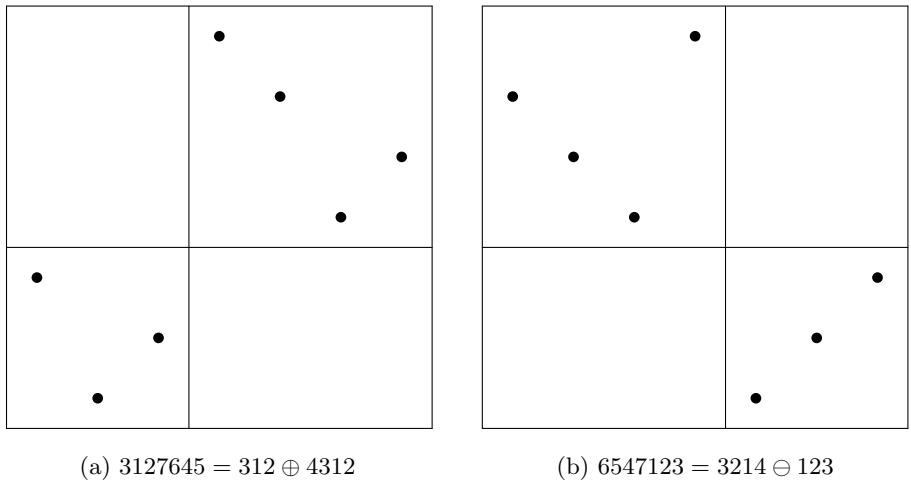

In addition, we will sometimes write $\pi_1 \oplus \pi_2 \oplus \cdots \oplus \pi_k$ as $\bigoplus_{i=1}^k\pi_i$.
 
\subsection{Permutation classes}

Two sequences of numbers $s_1, s_2, \ldots, s_n$ and $r_1, r_2, \ldots, r_n$ are
\emph{order-isomorphic} if for any two indices $i,j \in [n]$ it holds that $s_i < s_j$ 
if and only if $r_i < r_j$.

We define the following partial ordering on the set of all permutations.
We say that $\pi$ is \emph{contained in} $\sigma$ and write $\pi \leq \sigma$ if 
$\sigma$ has a~subsequence of length $|\pi|$ order-isomorphic to $\pi$.
See the example of containment in Figure \ref{fig_contain}. On the other hand,
if $\pi \nleq \sigma$, we say that $\sigma$ \emph{avoids} $\pi$.

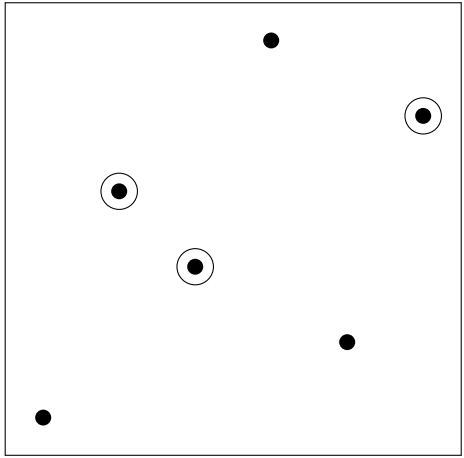
\begin{figure}[ht] 
  \centering
  \begin{tikzpicture}[line cap=round,line join=round,>=triangle 45,x=0.5cm,y=0.5cm]
    \clip(0.69,0.64) rectangle (13.4,13.35);
    \draw (1,1)-- (1,13);
    \draw (1,1)-- (13,1);
    \draw (13,1)-- (13,13);
    \draw (13,13)-- (1,13);
    \draw(4,8) circle (0.24cm);
    \draw(6,6) circle (0.24cm);
    \draw(12,10) circle (0.24cm);
    \begin{scriptsize}
    \fill [color=black] (2,2) circle (3pt);
    \fill [color=black] (4,8) circle (3pt);
    \fill [color=black] (6,6) circle (3pt);
    \fill [color=black] (8,12) circle (3pt);
    \fill [color=black] (10,4) circle (3pt);
    \fill [color=black] (12,10) circle (3pt);
    \end{scriptsize}
  \end{tikzpicture}
  \caption{The permutation 213 is contained in 143625.}
  \label{fig_contain}
\end{figure}

A set $\C$ of permutations is called a~\emph{permutation class} if for every $\pi \in \C$ and every $\sigma \leq \pi$
we have $\sigma \in \C$. We say that $\C$ \emph{avoids} a~permutation $\sigma$ if $\sigma \notin \C$, i.e. 
every $\pi \in \C$ avoids $\sigma$.
Permutation classes are often described by the patterns they avoid. If $B$ is any set of permutations,
we denote by $\Av(B)$ the set of all permutations avoiding every element of $B$.
Observe that $\C$ is a~permutation class if and only if $\C = \Av(B)$ for some set $B$. 
Indeed, if $\C$ is a~permutation class, then $\C = \Av(\S \setminus \C)$,
and if $\sigma \leq \pi \in \C$, then $\pi$ avoids all permutations of $B$ and clearly $\sigma$ avoids them too.
If $\C = \Av(B)$ and $B$ is an anti-chain with respect to containment, we call $B$ the \emph{basis} of $\C$.
Also if $B = \{\pi_1, \pi_2, \ldots, \pi_k\}$ is finite, we write just $\Av(\pi_1, \ldots, \pi_k)$
instead of $\Av(\{\pi_1, \ldots, \pi_k\})$. Finally, if $\C = \Av(\pi)$ for a~single
permutation $\pi$, we say that $\C$ is a~\emph{principal class}.

Let $s_1, s_2, \ldots, s_k$ be $k$ finite sequences of numbers. We denote their concatenation
by $s_1s_2\cdots s_k$. 
If a~sequence $s$ can be constructed by interleaving $s_1, s_2, \ldots, s_k$ in some (not necessarily unique) way,
we say that $s$ is a~\emph{merge of} or it is \emph{merged from} $s_1, s_2, \ldots, s_k$.

We define $\I_k$ resp. $\D_k$ to be the class of all permutations merged from at most $k$ increasing resp. 
decreasing subsequences. Also let $\I = \I_1$ and $\D=\D_1$, i.e. $\I = \Av(21)$ is the set of all increasing
permutations and $\D = \Av(12)$ is the set of all decreasing permutations, and for convenience let
$\I_0 = \D_0 = \S_0$. 

The classes $\I_k$ and $\D_k$ are well-known examples of principal classes.

\begin{fact}[Vatter \cite{Vatter15}] \label{basicfact}
$\I_{k-1} = \Av(k\cdots21)$ and $\D_{k-1} = \Av(12\cdots k)$ for any positive integer $k$.
\qed
\end{fact}

Next we recall a~known and important property of infinite permutation classes which will become
useful in the upcoming sections.

\begin{fact}[Atkinson, Beals \cite{AtkinsonBeals01}] \label{fact_infinite}
Let $\C$ be an infinite permutation class. Then either $\I \subseteq \C$ or $\D \subseteq \C$.
\qed
\end{fact}

\subsection{Splittability}

In this section we shortly introduce another concept which has been recently used
to derive enumerative results on permutation classes and which we will also utilize in our work.

A permutation $\pi$ is \emph{merged from permutations} $\alpha$ and $\beta$ if we can
color the elements of $\pi$ with red and blue such that the red subsequence is order-isomorphic
to $\alpha$ and the blue sequence is order-isomorphic to $\beta$. Given two 
permutation classes $\A$ and $\B$ we define their \emph{merge} denoted by $\A \odot \B$
as the class of all permutations which can be merged from a~(possibly empty) permutation from $\A$
and a~(possibly empty) permutation from $\B$. For example, it is easy to see that
\begin{equation*} \I_k = \underbrace{\I \odot \I \odot \cdots \odot \I}_{k\times}. \end{equation*}

We say that a~class $\C$ is \emph{splittable} if it has two proper subclasses $\A$ and $\B$ such that $\C \subseteq \A \odot \B$.
We refer the reader to the work of Jelínek and Valtr \cite{JelinekValtr13} for an exhaustive study of splittability.

\section{The notion of composability} \label{ch2}

In the following sections we provide definitions of the key notions of this work
as well as basic facts and observations.

\subsection{Composing permutation classes}

We define the \emph{composition} of two permutation classes $\A$ and $\B$
as the set $\A \c \B = \{ \pi \c \varphi; \pi \in \A, \varphi \in \B, |\pi| = |\varphi|\}$. 

\begin{lemma} Let $\A$ and $\B$ be arbitrary permutation classes.

\begin{enumerate}[(a)]
\item  $\A \c \B$ is again a~permutation class.
\item Composing permutation classes is associative, i.e. $(\A \c \B) \c \C = \A \c (\B \c \C)$.
\end{enumerate}

\end{lemma}
\begin{proof}
Let $\alpha \c \beta = \pi \in \A \c \B$, so that $\alpha \in \A$ and $\beta \in \B$.
Then a~permutation contained in $\pi$ at indices $i_1 < \cdots < i_r$ is composed
of $\alpha' \leq \alpha$ and $\beta' \leq \beta$ such that $\beta'$ is contained at indices
$i_1, \ldots, i_r$ in $\beta$ and $\alpha'$ is contained at indices $\beta(i_1), \ldots, \beta(i_2)$ in $\alpha$.
Associativity follows from associativity of permutation composition.
\end{proof}

Having verified associativity of the composition operator we can now define the composition of more than two classes in a~natural 
inductive way:
\begin{equation*} \C_1 \c \C_2 \c \cdots \c \C_k = (\C_1 \c \C_2 \c \cdots \c \C_{k-1}) \c \C_k. \end{equation*}
We will also sometimes use the power notation $\underbrace{\C \c \C \c \cdots \c \C}_{k \times} = (\C)^k$.

We continue by proving several simple lemmas about composing permutations merged from few increasing sequences.
\begin{lemma} \label{lemma_kl} $\I_k \c \I_l \subseteq \I_{kl}$ for any integers $k,l \geq 0$.
\end{lemma}
\begin{proof}
Choose $\pi \in \I_k$ and $\varphi \in \I_l$, partition $\varphi$ into $l$ increasing sequences and choose
one of them at indices $i_1 < \cdots < i_r$. Then $\varphi(i_1) < \cdots < \varphi(i_r)$ and so
$\pi(\varphi(i_1)), \ldots, \pi(\varphi(i_r))$ is a~subsequence of $\pi$ and therefore it can be partitioned
into at most $k$ increasing sequences since that is the property of $\pi$. This is true for the image of
each of the $l$ increasing subsequences in $\varphi$ and therefore $\pi \c \varphi$ can
be partitioned into at most $k\cdot l$ increasing subsequences.
\end{proof}

Since $\D \c \D = \I$, the argument of the previous proof can be repeated to show that $\D_k \c \D_l \subseteq \I_{kl}$.
We can generalise this even more.
\begin{lemma} \label{lemma_extrakl} Let $k,l,m,n$ be any non-negative integers. Then
\begin{equation*} (\I_k \odot \D_m) \c (\I_l \odot \D_n) \subseteq \I_{kl + mn} \odot \D_{kn + ml}. \end{equation*}
\end{lemma}
\begin{proof}
Use the approach identical to that of Lemma \ref{lemma_kl}.
\end{proof}

\subsection{Composability}

The main problem we are addressing in this work is whether permutations in a~given permutation class
can be constructed by composing permutations from two or more smaller classes. We formalise this as follows.
A permutation class $\C$ is said to be \emph{composable from classes} $\C_1, \ldots, \C_k$ if
$\C \subseteq \C_1 \c \cdots \c \C_k$. A~class $\C$ is \emph{$k$-composable}, if it is composable from its $k$
proper subclasses $\C_1, \ldots, \C_k$. A~class $\C$ is \emph{composable}, if it is $k$-composable for some $k \geq 2$.
Using this terminology, our goal is thus answering the question whether a~given permutation class is composable.

Clearly, for every class $\C$ we have $\C \subseteq \C \c \I$. For an infinite class we have either $\I \subseteq \C$,
which implies $\C \subseteq \C \c \C$, or $\D \subseteq \C$, which implies $\I \subseteq \C \c \C$ and $\C \subseteq \C \c \C \c \C$.
Restricting ourselves to proper subclasses in the definition of a~composable class is motivated
by these trivial inclusions.

We begin the exploration of composability by proving the following result which implies that unlike splittability,
$k$-composability for $k > 2$ does not imply $2$-composability.

\begin{thm} \label{evencomposable}
Let $\C$ be an infinite permutation class such that $\I \nsubseteq \C$. Then $\C$
is not $2k$-composable for any positive integer $k$.
\end{thm}

\begin{proof}
Since $\C$ does not contain $\I$, there is an integer $n$ such that $\C$ avoids $12\cdots n(n+1)$
and therefore $\C \subseteq \D_n$ by Fact \ref{basicfact}.

Now let $\A_1, \B_1, \A_2, \B_2, \ldots, \A_k, \B_k$ be proper subclasses of $\C$
and suppose that $\C \subseteq \A_1 \c \B_1 \c \A_2 \c \B_2 \c \cdots \c \A_k \c \B_k$.
Since all these classes are subsets of $\D_n$, Lemma \ref{lemma_extrakl}
implies $\A_i \c \B_i \subseteq \I_{n^2}$ for every $i \in [k]$.
Using Lemma \ref{lemma_extrakl} again we get that 
  
\begin{equation*}\A_1 \c \B_1 \c \A_2 \c \B_2 \c \cdots \c \A_k \c \B_k
\subseteq \I_{n^2} \c \cdots \c \I_{n^2} \subseteq \I_{n^{2k}},\end{equation*}
therefore, according to our assumption, $\C \subseteq \I_{n^{2k}}$, which means that $\C$ 
does not contain a~decreasing permutation of length $n^{2k}+1$ by Fact \ref{basicfact}. 
But since $\C$ is infinite and does not contain $\I$, it has to contain $\D$
according to Fact \ref{fact_infinite}, which is a~contradiction.
\end{proof}

\subsection{Properties of symmetries}
In this section we explore how composability is preserved under
some of the usual symmetrical maps.

For a~permutation $\pi$ of length $n$ we define $\pi^r$ to be the \emph{reverse} of $\pi$, i.e. $\pi^r(k) = \pi(n-k+1)$,
and $\pi^c$ to be the \emph{complement} of $\pi$, i.e. $\pi^c(k) = n-\pi(k)+1$.
For a~permutation class $\A$ we define 
the \emph{inverse class} 
$\A^{-1} = \{\pi^{-1}; \pi \in \A\}$, 
the \emph{reverse class} $\A^r = \{\pi^r; \pi \in \A\}$,  and the \emph{complementary class} $\A^c = \{\pi^c; \pi \in \A\}$.

\begin{figure}[h]
\centering
\subfloat[][$14352$] {
\begin{tikzpicture}[line cap=round,line join=round,>=triangle 45,x=0.6cm,y=0.6cm]
\clip(-0.18,-0.72) rectangle (5.14,4.7);
\draw (0,4.5)-- (0,-0.5);
\draw (0,-0.5)-- (5,-0.5);
\draw (5,-0.5)-- (5,4.5);
\draw (5,4.5)-- (0,4.5);
\draw (5.5,4.5)-- (5.5,-0.5);
\draw (5.5,-0.5)-- (10.5,-0.5);
\draw (10.5,-0.5)-- (10.5,4.5);
\draw (10.5,4.5)-- (5.5,4.5);
\draw (11,4.5)-- (11,-0.5);
\draw (11,-0.5)-- (16,-0.5);
\draw (16,-0.5)-- (16,4.5);
\draw (16,4.5)-- (11,4.5);
\draw (16.5,4.5)-- (16.5,-0.5);
\draw (16.5,-0.5)-- (21.5,-0.5);
\draw (21.5,-0.5)-- (21.5,4.5);
\draw (21.5,4.5)-- (16.5,4.5);
\begin{scriptsize}
\fill [color=black] (0.5,0) circle (2.0pt);
\fill [color=black] (1.5,3) circle (2.0pt);
\fill [color=black] (2.5,2) circle (2.0pt);
\fill [color=black] (3.5,4) circle (2.0pt);
\fill [color=black] (4.5,1) circle (2.0pt);
\fill [color=black] (6,0) circle (2.0pt);
\fill [color=black] (8,2) circle (2.0pt);
\fill [color=black] (9,1) circle (2.0pt);
\fill [color=black] (7,4) circle (2.0pt);
\fill [color=black] (10,3) circle (2.0pt);
\fill [color=black] (15.5,0) circle (2.0pt);
\fill [color=black] (14.5,3) circle (2.0pt);
\fill [color=black] (13.5,2) circle (2.0pt);
\fill [color=black] (12.5,4) circle (2.0pt);
\fill [color=black] (11.5,1) circle (2.0pt);
\fill [color=black] (17,4) circle (2.0pt);
\fill [color=black] (17.5,1) circle (2.0pt);
\fill [color=black] (19,2.08) circle (2.0pt);
\fill [color=black] (20,0) circle (2.0pt);
\fill [color=black] (21,3) circle (2.0pt);
\end{scriptsize}
\end{tikzpicture}
}
\subfloat[][$(14352)^{-1}$] {
\begin{tikzpicture}[line cap=round,line join=round,>=triangle 45,x=0.6cm,y=0.6cm]
\clip(5.36,-0.72) rectangle (10.7,4.7);
\draw (0,4.5)-- (0,-0.5);
\draw (0,-0.5)-- (5,-0.5);
\draw (5,-0.5)-- (5,4.5);
\draw (5,4.5)-- (0,4.5);
\draw (5.5,4.5)-- (5.5,-0.5);
\draw (5.5,-0.5)-- (10.5,-0.5);
\draw (10.5,-0.5)-- (10.5,4.5);
\draw (10.5,4.5)-- (5.5,4.5);
\draw (11,4.5)-- (11,-0.5);
\draw (11,-0.5)-- (16,-0.5);
\draw (16,-0.5)-- (16,4.5);
\draw (16,4.5)-- (11,4.5);
\draw (16.5,4.5)-- (16.5,-0.5);
\draw (16.5,-0.5)-- (21.5,-0.5);
\draw (21.5,-0.5)-- (21.5,4.5);
\draw (21.5,4.5)-- (16.5,4.5);
\begin{scriptsize}
\fill [color=black] (0.5,0) circle (2.0pt);
\fill [color=black] (1.5,3) circle (2.0pt);
\fill [color=black] (2.5,2) circle (2.0pt);
\fill [color=black] (3.5,4) circle (2.0pt);
\fill [color=black] (4.5,1) circle (2.0pt);
\fill [color=black] (6,0) circle (2.0pt);
\fill [color=black] (8,2) circle (2.0pt);
\fill [color=black] (9,1) circle (2.0pt);
\fill [color=black] (7,4) circle (2.0pt);
\fill [color=black] (10,3) circle (2.0pt);
\fill [color=black] (15.5,0) circle (2.0pt);
\fill [color=black] (14.5,3) circle (2.0pt);
\fill [color=black] (13.5,2) circle (2.0pt);
\fill [color=black] (12.5,4) circle (2.0pt);
\fill [color=black] (11.5,1) circle (2.0pt);
\fill [color=black] (17,4) circle (2.0pt);
\fill [color=black] (17.5,1) circle (2.0pt);
\fill [color=black] (19,2.08) circle (2.0pt);
\fill [color=black] (20,0) circle (2.0pt);
\fill [color=black] (21,3) circle (2.0pt);
\end{scriptsize}
\end{tikzpicture}
}
\subfloat[][$(14352)^{r}$] {
\begin{tikzpicture}[line cap=round,line join=round,>=triangle 45,x=0.6cm,y=0.6cm]
\clip(10.92,-0.72) rectangle (16.16,4.7);
\draw (0,4.5)-- (0,-0.5);
\draw (0,-0.5)-- (5,-0.5);
\draw (5,-0.5)-- (5,4.5);
\draw (5,4.5)-- (0,4.5);
\draw (5.5,4.5)-- (5.5,-0.5);
\draw (5.5,-0.5)-- (10.5,-0.5);
\draw (10.5,-0.5)-- (10.5,4.5);
\draw (10.5,4.5)-- (5.5,4.5);
\draw (11,4.5)-- (11,-0.5);
\draw (11,-0.5)-- (16,-0.5);
\draw (16,-0.5)-- (16,4.5);
\draw (16,4.5)-- (11,4.5);
\draw (16.5,4.5)-- (16.5,-0.5);
\draw (16.5,-0.5)-- (21.5,-0.5);
\draw (21.5,-0.5)-- (21.5,4.5);
\draw (21.5,4.5)-- (16.5,4.5);
\begin{scriptsize}
\fill [color=black] (0.5,0) circle (2.0pt);
\fill [color=black] (1.5,3) circle (2.0pt);
\fill [color=black] (2.5,2) circle (2.0pt);
\fill [color=black] (3.5,4) circle (2.0pt);
\fill [color=black] (4.5,1) circle (2.0pt);
\fill [color=black] (6,0) circle (2.0pt);
\fill [color=black] (8,2) circle (2.0pt);
\fill [color=black] (9,1) circle (2.0pt);
\fill [color=black] (7,4) circle (2.0pt);
\fill [color=black] (10,3) circle (2.0pt);
\fill [color=black] (15.5,0) circle (2.0pt);
\fill [color=black] (14.5,3) circle (2.0pt);
\fill [color=black] (13.5,2) circle (2.0pt);
\fill [color=black] (12.5,4) circle (2.0pt);
\fill [color=black] (11.5,1) circle (2.0pt);
\fill [color=black] (17,4) circle (2.0pt);
\fill [color=black] (17.5,1) circle (2.0pt);
\fill [color=black] (19,2.08) circle (2.0pt);
\fill [color=black] (20,0) circle (2.0pt);
\fill [color=black] (21,3) circle (2.0pt);
\end{scriptsize}
\end{tikzpicture}
}
\subfloat[][$(14352)^{c}$] {
\begin{tikzpicture}[line cap=round,line join=round,>=triangle 45,x=0.6cm,y=0.6cm]
\clip(16.38,-0.72) rectangle (21.82,4.7);
\draw (0,4.5)-- (0,-0.5);
\draw (0,-0.5)-- (5,-0.5);
\draw (5,-0.5)-- (5,4.5);
\draw (5,4.5)-- (0,4.5);
\draw (5.5,4.5)-- (5.5,-0.5);
\draw (5.5,-0.5)-- (10.5,-0.5);
\draw (10.5,-0.5)-- (10.5,4.5);
\draw (10.5,4.5)-- (5.5,4.5);
\draw (11,4.5)-- (11,-0.5);
\draw (11,-0.5)-- (16,-0.5);
\draw (16,-0.5)-- (16,4.5);
\draw (16,4.5)-- (11,4.5);
\draw (16.5,4.5)-- (16.5,-0.5);
\draw (16.5,-0.5)-- (21.5,-0.5);
\draw (21.5,-0.5)-- (21.5,4.5);
\draw (21.5,4.5)-- (16.5,4.5);
\begin{scriptsize}
\fill [color=black] (0.5,0) circle (2.0pt);
\fill [color=black] (1.5,3) circle (2.0pt);
\fill [color=black] (2.5,2) circle (2.0pt);
\fill [color=black] (3.5,4) circle (2.0pt);
\fill [color=black] (4.5,1) circle (2.0pt);
\fill [color=black] (6,0) circle (2.0pt);
\fill [color=black] (8,2) circle (2.0pt);
\fill [color=black] (9,1) circle (2.0pt);
\fill [color=black] (7,4) circle (2.0pt);
\fill [color=black] (10,3) circle (2.0pt);
\fill [color=black] (15.5,0) circle (2.0pt);
\fill [color=black] (14.5,3) circle (2.0pt);
\fill [color=black] (13.5,2) circle (2.0pt);
\fill [color=black] (12.5,4) circle (2.0pt);
\fill [color=black] (11.5,1) circle (2.0pt);
\fill [color=black] (17,4) circle (2.0pt);
\fill [color=black] (18,1) circle (2.0pt);
\fill [color=black] (19,2) circle (2.0pt);
\fill [color=black] (20,0) circle (2.0pt);
\fill [color=black] (21,3) circle (2.0pt);
\end{scriptsize}
\end{tikzpicture}
}
\caption{Symmetries of the permutation 14352}
\end{figure}
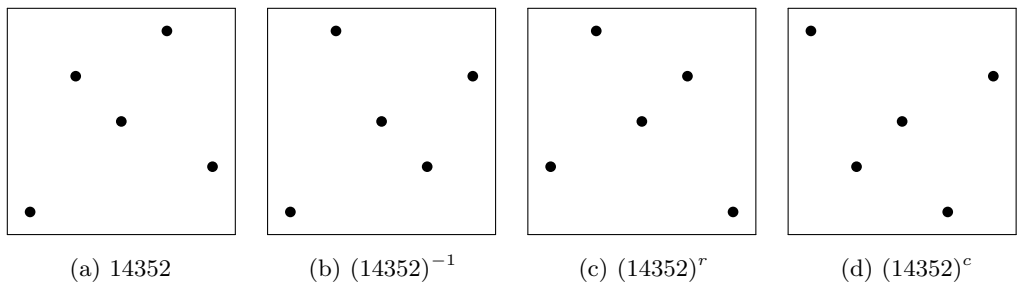

It is clear that all these class operators are involutory, i.e. $(\A^{-1})^{-1} =  \A$, $(\A^r)^r = \A$
and $(\A^c)^c = \A$. The following simple lemma describes how these operators relate to composition.

\begin{lemma} \label{lemma_basicsym}
Let $\A, \A_1, \A_2, \ldots, \A_k$ be permutation classes. Then
\begin{enumerate}[(a)]
\item $(\A_1 \c \A_2 \c \cdots \c \A_k)^{-1} = \A_k^{-1} \c \cdots \c \A_2^{-1} \c \A_1^{-1}$,
\item $\A^r = \A \c \D$ and $\A^c = \D \c \A$,
\item $(\A^r)^c = (\A^c)^r = \D \c \A \c \D$.
\end{enumerate}
\end{lemma}
\begin{proof}
(a): If $\pi_i \in \A_i$ for every $i \in [k]$, then by the property of inverse elements in a~group we have
$(\pi_1 \c \pi_2 \c \cdots \c \pi_k)^{-1} = \pi_k^{-1} \c \cdots \c \pi_2^{-1} \c \pi_1^{-1}$.

(b): Let $\alpha \in \A$ and $\delta \in \D$ be permutations of order $n$. By definition $\delta(k) = n - k + 1$
for every $k \in [n]$. Therefore $\alpha(\delta(k)) = \alpha(n-k+1) = \alpha^r(k)$ and 
  $\delta(\alpha(k)) = n - \alpha(k) + 1 = \alpha^c(k)$ for every $k \in [n]$.

(c):  Apparent from (b).
\end{proof}

Using this lemma we derive several composability criteria for symmetries of a~given class,
the first of which requires no further proof as it is an immediate consequence of Lemma \ref{lemma_basicsym}.
\begin{cor} \label{cor_inversion}
Let $\A$ be a~permutation class. Then the following statements are equivalent:
\begin{enumerate}[(a)]
\item $\A$ is composable,
\item $\A^{-1}$ is composable,
\item $(\A^r)^c$ is composable.
\end{enumerate}
\end{cor}

The case of the reverse and complementary operators is more complicated and requires additional assumptions.
\begin{lemma}\label{lemma_symcomp}
If $\A$ is a~$k$-composable class and $\I \subsetneq \A$, then both $\A^r$ and $\A^c$ are $(2k-1)$-composable.
\end{lemma}
\begin{proof}
Let $\A$ be composable from its proper subclasses $\A_1, \A_2, \ldots, \A_k$.
Then
\begin{equation*} \A^r = \A \c \D \subseteq \A_1 \c \A_2 \c \cdots \c \A_k \c \D =
(\A_1^r \c \D) \c (\A_2^r \c \D) \c \cdots \c (\A_k^r \c \D) \c \D.\end{equation*}
It holds that $\D \c \D = \I$, so we have
\begin{equation*} \A^r \subseteq \A_1^r \c \D \c \A_2^r \c \D \c \cdots \c \A_k^r.\end{equation*}
Clearly $\A_i^r \subsetneq \A^r$ and since $\I \subsetneq \A$, we have $\D \subsetneq \A^r$,
so the proper subclass criterion is met and $\A^r$ is therefore $(2k-1)$-composable.
Analogously we show that
\begin{equation*} \A^c \subseteq \A_1^c \c \D \c \A_2^c \c \D \c \cdots \c \A_k^c. \qedhere \end{equation*}
\end{proof}

\section{On permutations avoiding a decreasing sequence} \label{ch3}

Recall that $\I_k = \Av((k+1)\cdots21)$ is the class of permutations merged from $k$ increasing sequences, 
or equivalently those avoiding a~decreasing sequence of length $k+1$.
In this section, we prove that $\I_k$ is 2-composable and show several
examples of how $\I_k$ can be composed from two proper subclasses. 
      
\subsection{Vertical and horizontal merge}

Let $\C_1, \ldots, \C_k$ be any permutation classes. We define the \emph{vertical merge} of these classes as the class 
of permutations that can be written as a~concatenation $s_1s_2\cdots s_k$ of $k$ (possibly empty) sequences 
such that $s_i$ is order-isomorphic
to a~permutation of $\C_i$. We write this class as $\V(\C_1, \ldots, \C_k)$. In addition,
if $\C_1 = \C_2 = \cdots = \C_k = \I$, we let $\V_k$ denote the class $\V(\C_1, \ldots, \C_k)$. Similarly we define
the $\emph{horizontal merge}$ of these classes as the class of permutations that
can be written as a~merge of $k$ (possibly empty) sequences $s_1, s_2, \ldots, s_k$
such that each $s_i$  is order-isomorphic to $\pi_i \in \C_i$ and every element of $s_i$
is smaller than every element of $s_{i+1}$ for $1 \leq i \leq k-1$. Note that this implies
that each $s_i$ uses a set of consecutive integers.
We let $\Ho(\C_1, \ldots, \C_k)$ denote the horizontal merge of classes $\C_1, \ldots, \C_k$ and if 
$\C_1 = \C_2 = \cdots = \C_k = \I$ we write $\Ho_k = \Ho(\C_1, \ldots, \C_k)$.

Alternatively, we can observe that $\pi \in \V(\C_1, \ldots, \C_k)$ resp. $\pi \in \Ho(\C_1, \ldots,\C_k)$ 
if and only if its plot in $\R^2$ can be separated by vertical resp.
horizontal lines into at most $k$ parts , $i$-th of them containing a~sequence order-isomorphic
to a~permutation in $\C_i$ (see Figure \ref{figure_VH}), hence the names of the classes.

\begin{figure}[h]
  \centering
  \subfloat[][An element of the vertical merge $\V_k$]{
    \begin{tikzpicture}[line cap=round,line join=round,>=triangle 45,x=1.0cm,y=1.0cm]
      \clip(1.66,0.44) rectangle (8.4,5.5);
      \draw (3,1)-- (7,1);
      \draw (7,1)-- (7,5);
      \draw (7,5)-- (3,5);
      \draw (3,5)-- (3,1);
      \draw (3.84,5)-- (3.84,1);
      \draw (4.86,1)-- (4.86,5);
      \draw (6,5)-- (6,1);
      \begin{scriptsize}
      \fill [color=black] (3.24,1.44) circle (1.5pt);
      \fill [color=black] (3.48,2.56) circle (1.5pt);
      \fill [color=black] (3.66,4.26) circle (1.5pt);
      \fill [color=black] (4.06,1.14) circle (1.5pt);
      \fill [color=black] (4.28,1.94) circle (1.5pt);
      \fill [color=black] (4.5,3.18) circle (1.5pt);
      \fill [color=black] (4.74,4.78) circle (1.5pt);
      \fill [color=black] (5.04,1.36) circle (1.5pt);
      \fill [color=black] (5.22,2.18) circle (1.5pt);
      \fill [color=black] (5.38,2.82) circle (1.5pt);
      \fill [color=black] (5.62,3.94) circle (1.5pt);
      \fill [color=black] (5.82,4.54) circle (1.5pt);
      \fill [color=black] (6.18,1.64) circle (1.5pt);
      \fill [color=black] (6.3,2.4) circle (1.5pt);
      \fill [color=black] (6.58,3.44) circle (1.5pt);
      \fill [color=black] (6.84,4.1) circle (1.5pt);
      \end{scriptsize}
    \end{tikzpicture}
  }
  \subfloat[][An element of the horizontal merge $\Ho_k$]{
    \begin{tikzpicture}[line cap=round,line join=round,>=triangle 45,x=1.0cm,y=1.0cm]
      \clip(6.8,0.44) rectangle (13.22,5.24);
      \draw (6,5)-- (6,1);
      \draw (8,1)-- (12,1);
      \draw (12,1)-- (12,5);
      \draw (12,5)-- (8,5);
      \draw (8,5)-- (8,1);
      \draw (8,4.12)-- (12,4.12);
      \draw (8,2.96)-- (12,2.96);
      \draw (8,2.14)-- (12,2.14);
      \begin{scriptsize}
      \fill [color=black] (8.72,1.12) circle (1.5pt);
      \fill [color=black] (9.26,1.38) circle (1.5pt);
      \fill [color=black] (10.22,1.68) circle (1.5pt);
      \fill [color=black] (11.34,1.96) circle (1.5pt);
      \fill [color=black] (8.34,2.24) circle (1.5pt);
      \fill [color=black] (9.86,2.56) circle (1.5pt);
      \fill [color=black] (10.94,2.86) circle (1.5pt);
      \fill [color=black] (8.1,3.06) circle (1.5pt);
      \fill [color=black] (8.96,3.42) circle (1.5pt);
      \fill [color=black] (10.42,3.68) circle (1.5pt);
      \fill [color=black] (11.88,3.98) circle (1.5pt);
      \fill [color=black] (8.62,4.2) circle (1.5pt);
      \fill [color=black] (9.66,4.48) circle (1.5pt);
      \fill [color=black] (11.24,4.78) circle (1.5pt);
      \end{scriptsize}
    \end{tikzpicture}   
  }
\caption{Examples of vertical and horizontal merges}
\label{figure_VH}
\end{figure}
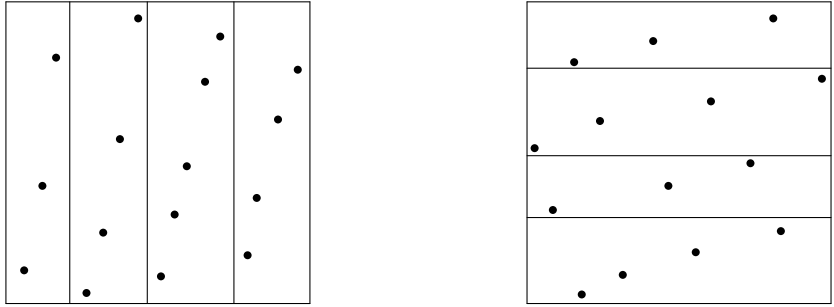
In addition we define $\Ho = \Ho_2$ and $\V = \V_2$ for future convenience. 
We continue by observing an important connection between the horizontal and vertical merge.

\begin{lemma}\label{lemma_VHinvert}
Let $\C_1, \ldots, \C_k$ be any permutation classes. Then
\begin{equation*} \Ho(\C_1, \ldots, \C_k) = \left(\V(\C_1^{-1}, \ldots, \C_k^{-1})\right)^{-1}. \end{equation*}
\end{lemma}
\begin{proof}

If $\pi \in \V(\C^{-1}_1, \ldots, \C^{-1}_k)$, we have that $\pi = s_1s_2\cdots s_k$ such that $s_i$ is order-isomorphic
to $\pi_i \in \C_i^{-1}$. For every $i \in [k]$, $\pi^{-1}$ contains a~set of consecutive integers
on indices $(s_i)_1, (s_i)_2, \ldots, (s_i)_{|s_i|}$ and the sequence at these indices is order-isomorphic
 to $\pi_i^{-1} \in \C_i$.

The opposite inclusion is equally straightforward.
\end{proof}

When composed with any other class $\A$, the classes $\Ho_k$, $\V_k$ and $\I_k$ can be viewed 
as a~unary operator transforming $\A$ in a~specific way.
We formalise this approach in the following lemma.
\begin{lemma} \label{lemma_behaviour}
Let $\A$ be an arbitrary permutation class. Then
\begin{enumerate}[(a)]
\item $\A \c \Ho_k$ is precisely the class of permutations which can be obtained from a~permutation of $\A$
  by dividing it into at most $k$ contiguous subsequences and interleaving them in any way,
\item $\A \c \V_k$ is precisely the class of permutations which can be obtained from a~permutation of $\A$
by dividing it into at most $k$ subsequences and concatenating them,
\item $\A \c \I_k$ is precisely the class of permutations which can be obtained from a~permutation of $\A$
by dividing it into at most $k$ subsequences and interleaving them in any way.
\end{enumerate}
\end{lemma}

\begin{proof}
Let $\alpha \in \A$, $\eta \in \Ho_k$, $\nu \in \V_k$ and $\iota \in \I_k$.

(a):
Consider the permutation $\alpha \c \eta \in \A \c \Ho_k$. Then $\eta$ is merged from $k$ (possibly empty) sequences
of consecutive integers $s_1, \ldots, s_k$. Now we define $k$ sequences $r_1, \ldots, r_k$
such that $|s_i| = |r_i|$ and $(r_i)_j = \alpha((s_i)_j)$ for every $i \in [k]$ and $j \in [|s_i|]$
. Every $r_i$ is a~contiguous subsequence of $\alpha$ and at the same time $\alpha \c \eta$
is merged from $r_1, \ldots, r_k$.

On the other hand, if a~permutation $\pi$ is obtained from $\alpha \in \A$ by dividing it into $k$~contiguous subsequences
$r_1, \ldots, r_k$ and merging them in some way, we define $k$~sequences $s_1, \ldots, s_k$
such that $s_i$ is the sequence of indices of the elements of $r_i$ in $\alpha$. Then by definition $\alpha((s_i)_j) = (r_i)_j$
for any suitable $i$~and $j$, and since we divided $\alpha$
into contiguous subsequences, each $s_i$ is a~sequence of consecutive integers.
Now consider the permutation $\eta$ created by replacing the subsequence $r_i$ by the sequence $s_i$ in $\pi$ for every $i$.
Then $\eta$ is merged from $s_1, \ldots, s_k$, which are sequences of consecutive integers, therefore $\eta \in \Ho$.
At the same time, for any $m \in \{1, 2, \ldots, |\pi|\}$ there are indices $i$~and $j$~such that $\pi(m) = (r_i)_j =
\alpha((s_i)_j) = \alpha(\eta(m))$, where the last equality holds because we replaced $(r_i)_j$ by $(s_i)_j$ when
constructing $\eta$ from $\pi$. Therefore $\pi = \alpha \c \eta$.


(b):
Consider the permutation $\alpha \c \nu \in \A \c \V_k$. The permutation $\nu$ is formed by concatenating $k$ increasing sequences
$s_1, \ldots, s_k$. Define $k$ sequences $r_1, \ldots, r_k$ such that $|r_i| = |s_i|$ and $(r_i)_j= \alpha((s_i)_j)$.
Each $r_i$ is a~subsequence of $\alpha$ and at the same time $\alpha \c \nu = r_1r_2\cdots r_k$.
 
On the other hand, if a~permutation $\pi$ is obtained from $\alpha \in \A$ by dividing it into $k$ subsequences
$r_1, \ldots, r_k$ and then concatenating them, we define $k$ sequences $s_1, \ldots, s_k$ such that
$s_i$ is the sequence of indices of elements of $r_i$ in $\alpha$. Thus every $s_i$ is an increasing sequence
and $\alpha((s_i)_j) = (r_i)_j$.
Consider a permutation $\nu$ created by replacing the subsequence $r_i$ by the sequence $s_i$ in $\pi$ for every $i$.
Since $\pi$ is a concatenation of $r_1, \ldots, r_k$, we get that $\nu$ is a concatenation of $s_1, \ldots, s_k$
and thus $\nu \in \V_k$.
Also, for any $m \in \{1, 2, \ldots, |\pi|\}$ there are indices $i$~and $j$~such that $\pi(m) = (r_i)_j =
\alpha((s_i)_j) = \alpha(\nu(m))$, where the last equality holds because we replaced $(r_i)_j$ by $(s_i)_j$ when
constructing $\nu$ from $\pi$. Therefore $\pi = \alpha \c \nu$.

(c): The proof is similar to the proofs of (a) and (b).
%
\end{proof}


\subsection{Composability results}

Using the machinery introduced in the previous section
we  now prove a~key lemma which we will use to show several composability results.

\begin{lemma}\label{important_lemma}

Let $\C_1, \C_2, \ldots, \C_k$ be arbitrary permutation classes.
Then \begin{equation*}\C_1 \odot \C_2 \odot \cdots \odot \C_k \subseteq \V(\C_1, \ldots, \C_k) \c \Ho_k.\end{equation*}
\end{lemma}
\begin{proof}
Consider a~permutation $\pi \in \C_1 \odot \cdots \odot \C_k$ and divide
it into $k$ sequences $s_1, \ldots, s_k$ such that $s_i$ is isomorphic to a~permutation from $\C_i$.
The permutation $\nu = s_1s_2\cdots s_k$ then lies in $\V(\C_1, \ldots, \C_k)$,
which together with Lemma \ref{lemma_behaviour}(a) implies $\pi \in \V(\C_1, \ldots, \C_k) \c \Ho_k$.
\end{proof}

By reformulating the previous statement we immediately get the following.

\begin{cor}\label{cor_splitcompose}
Let $\A$, $\B$ and $\C$ be permutation classes such that $\C \subseteq \A \odot \B$. Then $\C \subseteq \V(\A,\B) \c \Ho$.
\end{cor}

Using what has already been shown in this section it is now elementary to show that $\I_k$ is 2-composable.

\begin{thm} \label{thm_Ik}
The class $\I_k$ is 2-composable for every $k \geq 2$. In particular, $\I_k \subseteq \V_k \c \Ho_k$.
\end{thm}
\begin{proof}
Trivially $\V_k \subsetneq \I_k$ and $\Ho_k \subsetneq \I_k$.
Next we recall that
\begin{equation*}\I_k = \underbrace{\I \odot \cdots \odot \I}_{k\times}\end{equation*} 
and use Lemma \ref{important_lemma}  for $\C_1 = \C_2 = \cdots = \C_k = \I$.
\end{proof}

We proceed by proving a~result in some sense opposite to that of Lemma \ref{lemma_kl}, namely
we show that $\I_k$ may be constructed from smaller $\I_a, \I_b$ using composition.
\begin{thm}\label{thm_k+l-1}
$\I_{k+l-1} \subseteq \I_k \c \I_l$ for all integers $k, l \geq 2$.
\end{thm}
\begin{proof}

Consider a~permutation $\pi \in \I_{k+l-1}$, merged from two sequences $a$ and $b$
such that $a$ is merged from $k$ increasing sequences $s_1, \ldots, s_k$ and $b$ is merged from $l-1$ increasing sequences 
$s_{k+1}, \ldots, s_{k+l-1}$. Let $c$ be the increasing sequence created by sorting the elements of $b$.
Consider a~permutation $\sigma$ created by merging the sequences $a$ and $c$ so that $c$ and $s_k$ form a~single
increasing sequence. Clearly $\sigma \in \I_k$ and sequences $s_{k+1}, \ldots, s_{k+l-1}$ are subsequences
of $\sigma$, since they are increasing and therefore were not affected by sorting $b$.

According to Lemma \ref{lemma_behaviour}(c) the class $\I_k \c \I_l$ contains all permutations we can create from $\sigma$ by dividing
it into $l$ subsequences and merging them in any way. It is therefore enough to find a~way to divide $\sigma$
into $l$ subsequences which can be merged into $\pi$. A~simple choice of $l$ such subsequences is
$a, s_{k+1}, \ldots, s_{k+l-1}$.
\end{proof}

This theorem raises the question whether we could construct a~bigger class from given $\I_k$ and $\I_l$.

\begin{question}
Given positive integers $k$ and $l$, what is the largest integer $m = m(k,l)$ such that $\I_m \subseteq \I_k \c \I_l$?
\end{question}

So far we have shown that $m(k,l) \leq kl$ (Lemma \ref{lemma_kl})  and that \\
$m(k,l) \geq k+l-1$ (Theorem \ref{thm_k+l-1}). It is also not difficult to show the sharp
inequality $m(k,l) < kl$ by constructing a permutation $\pi \in \I_{kl} \setminus (\I_k \c \I_l)$.

\section{Classes of layered patterns}
\label{ch4}

In this section we cover classes of permutations which can be written as a~sum or as a~skew sum
of increasing or decreasing permutations. Among these classes we  provide
infinitely many examples of composable classes as well as several examples of classes which are
not composable.

Let $\iota_k$ denote the increasing permutation of order $k$ and
$\delta_k$ denote the decreasing permutation of order $k$.
A permutation is \emph{layered} if it is a~sum of decreasing permutations
which are then called \emph{layers}. We let $\La$ denote the class of all layered permutations. We let $\La_k$ denote
the class of permutations which are sums of at most $k$ layers. 
The complement of a~layered permutation is clearly a~skew sum of increasing permutations and we 
call such a~permutation \emph{co-layered}. The class $\La^c$ consists of precisely the co-layered permutations.

\begin{figure}[h]
\centering
\subfloat[][Layered permutation] {
\begin{tikzpicture}[line cap=round,line join=round,>=triangle 45,x=0.5cm,y=0.5cm]
\clip(-4.12,-6.72) rectangle (8.12,5.62);
\draw [line width=1.2pt] (-4,5.5)-- (-4,-6.5);
\draw [line width=1.2pt] (-4,-6.5)-- (8,-6.5);
\draw [line width=1.2pt] (8,-6.5)-- (8,5.5);
\draw [line width=1.2pt] (8,5.5)-- (-4,5.5);
\draw [line width=1.2pt] (8.5,5.5)-- (8.5,-6.5);
\draw [line width=1.2pt] (8.5,-6.5)-- (20.5,-6.5);
\draw [line width=1.2pt] (20.5,-6.5)-- (20.5,5.5);
\draw [line width=1.2pt] (20.5,5.5)-- (8.5,5.5);
\begin{scriptsize}
\fill [color=black] (-3.5,-4) circle (2.0pt);
\fill [color=black] (-2.5,-5) circle (2.0pt);
\fill [color=black] (-1.5,-6) circle (2.0pt);
\fill [color=black] (-0.5,-2) circle (2.0pt);
\fill [color=black] (0.5,-3) circle (2.0pt);
\fill [color=black] (1.5,4) circle (2.0pt);
\fill [color=black] (2.5,3) circle (2.0pt);
\fill [color=black] (3.5,2) circle (2.0pt);
\fill [color=black] (4.5,1) circle (2.0pt);
\fill [color=black] (5.5,0) circle (2.0pt);
\fill [color=black] (6.5,-1) circle (2.0pt);
\fill [color=black] (7.5,5) circle (2.0pt);
\fill [color=black] (9,3) circle (2.0pt);
\fill [color=black] (10,4) circle (2.0pt);
\fill [color=black] (11,5) circle (2.0pt);
\fill [color=black] (12,-1) circle (2.0pt);
\fill [color=black] (13,0) circle (2.0pt);
\fill [color=black] (14,1) circle (2.0pt);
\fill [color=black] (15,2) circle (2.0pt);
\fill [color=black] (16,-3) circle (2.0pt);
\fill [color=black] (17,-2) circle (2.0pt);
\fill [color=black] (18,-6) circle (2.0pt);
\fill [color=black] (19,-5) circle (2.0pt);
\fill [color=black] (20,-4) circle (2.0pt);
\end{scriptsize}
\end{tikzpicture}
}
\subfloat[][Co-layered permutation] {
\begin{tikzpicture}[line cap=round,line join=round,>=triangle 45,x=0.5cm,y=0.5cm]
\clip(8.42,-6.72) rectangle (20.66,5.62);
\draw [line width=1.2pt] (-4,5.5)-- (-4,-6.5);
\draw [line width=1.2pt] (-4,-6.5)-- (8,-6.5);
\draw [line width=1.2pt] (8,-6.5)-- (8,5.5);
\draw [line width=1.2pt] (8,5.5)-- (-4,5.5);
\draw [line width=1.2pt] (8.5,5.5)-- (8.5,-6.5);
\draw [line width=1.2pt] (8.5,-6.5)-- (20.5,-6.5);
\draw [line width=1.2pt] (20.5,-6.5)-- (20.5,5.5);
\draw [line width=1.2pt] (20.5,5.5)-- (8.5,5.5);
\begin{scriptsize}
\fill [color=black] (-3.5,-4) circle (2.0pt);
\fill [color=black] (-2.5,-5) circle (2.0pt);
\fill [color=black] (-1.5,-6) circle (2.0pt);
\fill [color=black] (-0.5,-2) circle (2.0pt);
\fill [color=black] (0.5,-3) circle (2.0pt);
\fill [color=black] (1.5,4) circle (2.0pt);
\fill [color=black] (2.5,3) circle (2.0pt);
\fill [color=black] (3.5,2) circle (2.0pt);
\fill [color=black] (4.5,1) circle (2.0pt);
\fill [color=black] (5.5,0) circle (2.0pt);
\fill [color=black] (6.5,-1) circle (2.0pt);
\fill [color=black] (7.5,5) circle (2.0pt);
\fill [color=black] (9,3) circle (2.0pt);
\fill [color=black] (10,4) circle (2.0pt);
\fill [color=black] (11,5) circle (2.0pt);
\fill [color=black] (12,-1) circle (2.0pt);
\fill [color=black] (13,0) circle (2.0pt);
\fill [color=black] (14,1) circle (2.0pt);
\fill [color=black] (15,2) circle (2.0pt);
\fill [color=black] (16,-3) circle (2.0pt);
\fill [color=black] (17,-2) circle (2.0pt);
\fill [color=black] (18,-6) circle (2.0pt);
\fill [color=black] (19,-5) circle (2.0pt);
\fill [color=black] (20,-4) circle (2.0pt);
\end{scriptsize}
\end{tikzpicture}
}
\caption{Examples of layered and co-layered patterns}
\end{figure}
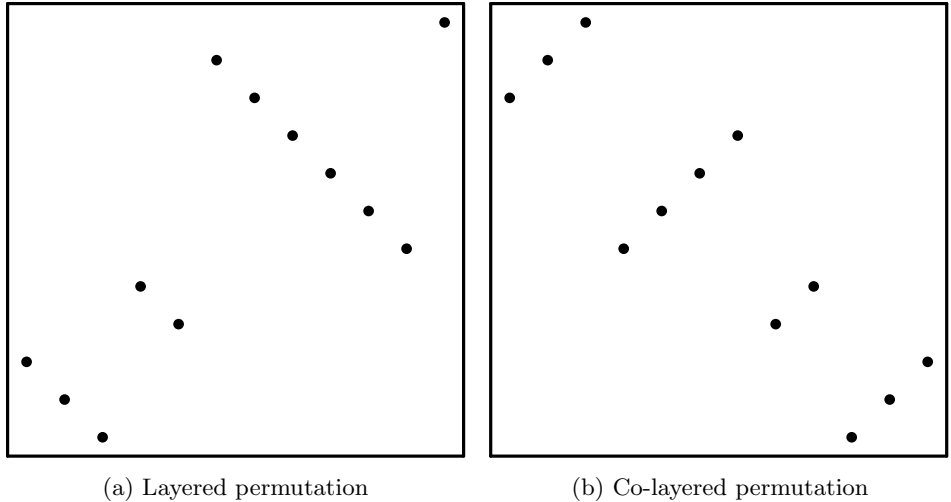

We start by proving that $\La_2$ is not composable using a~counting argument. As it turns out,
proper subclasses of $\La_2$ are asymptotically too small to build the entire $\La_2$ class
using composition.

\begin{thm} \label{l2}
The class $\La_2$ is not composable.
\end{thm}
\begin{proof}
Suppose that $\La_2 \subseteq \C_1 \c \C_2 \c \cdots \c \C_k$ such that $\C_i \subsetneq \La_2$ for every $i \in [k]$.
Each of these subclasses avoids at least one permutation of $\La_2$. In other words
for every $\C_i$ there is a~$\pi_i \in \La_2$ such that $\C_i \subseteq \La_2 \cap \Av(\pi_i)$. 
Considering a~sufficiently large $n$ so that 
$\pi_i \leq \delta_n \oplus \delta_n$ for every $i \in [k]$
we get that $\C_i \subseteq \La_2 \cap \Av(\delta_n \oplus \delta_n)$ for every $i$, in other words
every permutation in these subclasses has one of its two layers shorter than $n$.
It follows that for a~fixed integer $N$ there are at most $2(n-1)$ permutations of order $N$
in any $\C_i$, therefore there are at most $(2n-2)^k$ permutations in $\C_1 \c \cdots \c \C_k$.
But $\La_2$ contains $N$ permutations of order $N$ for any $N$, therefore we obtain
a contradiction by choosing $N > (2n-2)^k$.
\end{proof}

The number of permutations of order $n$ in $\La_2$ is linear in $n$ while any proper subclass contains
only constantly many permutations of fixed order. We can use the same approach using the asymptotic
jump from polynomial to exponential functions to show that a~different class of permutations cannot be composable.
Namely, let $\F_2$ be the class of layered permutations with layers of size 1~or 2.

\begin{thm}\label{f2}
The class $\F_2$ is not composable.
\end{thm}
\begin{proof}
Suppose $\F_2$ is composable from $k$ of its proper subclasses $\C_1, \C_2, \ldots, \C_k$.
We choose a~permutation from $\F_2 \setminus C_i$ for every $i$ and we select $n$ large enough so that
every chosen permutation is contained in $\pi = \sum_{i=1}^n\delta_2$. Then if $\C = \F_2 \cap \Av(\pi)$,
we get that $\F_2 \subseteq (\C)^k$. Every permutation in $\C$ contains fewer than $n$ layers of size 2, otherwise
it would contain $\pi$. Clearly there are at most $N^a$ permutations of $\F_2$ that have order $N$
and exactly $a$ layers of size 2. Therefore $\C$ contains at most $N^1 + N^2 + \cdots + N^{n-1} \leq nN^n$
permutations of order $N$ and the composition $(\C)^k$ then contains at most $n^kN^{nk}$ permutations of 
order $N$, which is a~number polynomial in $N$. As mentioned in \cite[Chapter 4]{Vatter15}, 
the number of permutations of order $N$ of $\F_2$
is counted by the Fibonacci numbers  which grow exponentially, therefore there is $N$ large enough so that 
$\F_2$ has more permutations of order $N$ than $(\C)^k$. 

Note that this result also follows immediately from the theorem of Kaiser and Klazar (\cite[3.4]{KaiserKlazar03}),
which states that if the number of permutations of order $n$ in a~permutation class is less than the $n$-th
Fibonacci number for at least one value of $n$, then it is eventually polynomial in $n$. This implies
that every class counted by the Fibonacci numbers is uncomposable.
\end{proof}

The argument used in the proofs above cannot be used for $\La_3$,
so we need a~different approach to show that this class too is not composable. We will make use of the following 
property of $\La_2 \cup \La_2^r$.

\begin{lemma} \label{l2group}
$(\La_2 \cup \La_2^r) \cap \S_n$ is a~subgroup of $\S_n$ for every $n$, i.e. it is closed under composition.
\end{lemma}
\begin{proof}
In this proof, we consider an additive group structure on the set $[n]$
with the neutral element $n$ and an operator $+_n$ defined as 
\begin{equation*} a +_n b = 1 + (a + b - 1)\mod n. \end{equation*}
First we prove that $\La_2^r \cap \S_n$ by itself is a~subgroup of $\S_n$.
Observe that $\La_2^r \cap \S_n$ contains exactly permutations $\pi$ such that there
is a~shifting number $k$ with $\pi(i) = i+_nk$ for every $i \in [n]$.
Indeed, if $\pi = \iota_a \ominus \iota_b$ then for any $i \in [n]$ we have 
$\pi(i) = i+_nb$ and conversely if $\pi(i) = i+_nk$
for every $i\in [n]$ then $\pi = \iota_{n-k} \ominus \iota_k$. Now for two permutations $\pi, \sigma \in \La^r_2$
with shifting numbers $k,l$ respectively we have $\pi(\sigma(i)) =
i +_n l +_n k$ for any $i \in [n]$, therefore $\pi \c \sigma \in \La_2^r$ since it has
a shifting number $k+_nl$.

It trivially holds that $\La_2 \c \D =  \La_2^r = \La_2^c = \D \c \La_2$. Considering
$\pi, \sigma \in (\La_2 \cup \La_2^r) \cap \S_n$ it remains to distinguish the following four cases:
\begin{enumerate}[(i)]
\item $\pi \in \La_2^r$ and $\sigma \in \La_2^r$, then $\pi \c \sigma \in \La_2^r$ by the discussion above,
\item $\pi \in \La_2^r$ and $\sigma \in \La_2$, then $\pi \c \sigma = (\pi \c \sigma^r) \c \delta_n \in \La_2$,
\item $\pi \in \La_2$ and $\sigma \in \La_2^r$, then $\pi \c \sigma = (\delta_n \c \pi^c) \c \sigma =
\delta_n \c (\pi^r \c \sigma) \in \La_2$,
\item $\pi \in \La_2$ and $\sigma \in \La_2$, then $\pi \c \sigma = \pi^r \c (\delta_n \c \delta_n)  \c \sigma^c =
\pi^r \c \sigma^r \in \La^r_2$.
\qedhere
\end{enumerate}
\end{proof}

\begin{thm}\label{l3}
The class $\La_3$ is not composable.
\end{thm}
\begin{proof}
Suppose that $\La_3 \subseteq \C_1 \c \C_2 \c \cdots \c \C_k$ such that $\C_i \subsetneq \La_3$ for any $i$.
Using the same initial argumentation as in the proof of Theorem \ref{l2} we get
that there is an $n$ such that $\La_3 \subseteq (\La_3 \cap \Av(\delta_n \oplus \delta_n \oplus \delta_n))^k$,
meaning that every permutation of $\La_3$ can be composed from $k$ permutations having at least one
of the three layers shorter than $n$.

Let $\pi_i \in \C_i$ for $1 \leq i \leq k$ and $\pi = \pi_1 \c \pi_2 \c \cdots \c \pi_k$.
We now claim that it is possible to remove at most $(n-1)k$ elements from $\pi$
to obtain a~two-layered or a~two-co-layered permutation. We will prove this by induction on $k$.
The case $k=1$ is easy since $\pi = \pi_1$ avoids $\delta_n \oplus \delta_n \oplus \delta_n$,
so it has a~layer of length shorter than $n$ whose removal creates a~two-layered pattern.

For $k > 1$ let $\sigma = \pi_1 \c \cdots \c \pi_{k-1}$ and $\pi = \sigma \c \pi_k$.
Let all these permutations have the order $N$.
By the induction hypothesis, there are $a$ indices $i_1, \ldots i_a$ such
that $a \geq N - (n-1)(k-1)$ and $\sigma$ restricted to these indices has the two-layer 
or the two-co-layer pattern. Also there are $b$ indices $j_{1}, \ldots, j_{b}$
such that $b \geq N - (n-1)$ and $\pi_k$ restricted to these indices forms the two-layer
or the two-co-layer pattern.

Let us now restrict the function $\sigma \c \pi_k$ to the set
$S = \{\pi_k^{-1}(i_1), \ldots, \pi_k^{-1}(i_a)\} \cap \{j_{1}, \ldots, j_{b}\}$ whose size is at least $N - (n-1)k$.
Then both $\pi_k(S)$ and $\sigma(\pi_k(S))$ are still two-layer or two-co-layer patterns,
which implies the same for their composition according to Lemma \ref{l2group}.
Therefore $\pi$ restricted to $S$ forms a~two-layer or two-co-layer pattern
and $N - |S| \leq (n-1)k$ which completes the induction step.

Consequently, any permutation of order $N$ in $\C_1 \c \C_2 \c \cdots \c \C_k$ 
contains a~two-layered or a~two-co-layered pattern of size at least  $N - k(n-1)$.
But choosing $N = 3(k(n-1)+1)$ and considering the permutation $\bigoplus_{i=1}^{3}\delta_{k(n-1)+1} \in \La_3$
we obtain a~contradiction.
\end{proof}

If we allow more than three but still constantly many layers, we always get a~composable class.

\begin{thm} \label{l4}
The class $\La_k$ is $3$-composable for every $k \geq 4$.
\end{thm}
\begin{proof}
We will show that $\La_k \subseteq \La_{k-1} \c \La_{k-2} \c \La_{k-1}$.

If $\pi \in \La_k$ of order $n$ has fewer than $k$ layers, then $\pi = \pi \c \delta_n \c \delta_n \in \La_{k-1} \c \La_{k-2} \c \La_{k-1}$.
Otherwise $\pi$ has at least 4~layers and has the form $\pi = \delta_a \oplus \delta_b \oplus \delta_c \oplus \delta_d \oplus \pi'$
for some positive $a$, $b$, $c$, $d$. Since for every layered $\sigma$ we have $\sigma \c \sigma \c \sigma = \sigma$
it is not hard to check that
\begin{equation*}\pi = (\delta_{a+b} \oplus \delta_c \oplus \delta_d \oplus \pi') \c
        (\delta_{a+b} \oplus \delta_{c+d} \oplus \pi') \c
        (\delta_a \oplus \delta_b \oplus \delta_{c+d} \oplus \pi') \in \La_{k-1} \c \La_{k-2} \c \La_{k-1}.
        \end{equation*}
The situation is represented in Figure \ref{figure_l4}.
\end{proof}

\begin{figure}[h]
\begin{tikzpicture}[line cap=round,line join=round,>=triangle 45,x=0.74cm,y=0.74cm]
\clip(2,1.8) rectangle (20.6,6.18);
\draw (2,2)-- (6,2);
\draw (6,2)-- (6,6);
\draw (6,6)-- (2,6);
\draw (2,6)-- (2,2);
\draw [dash pattern=on 1pt off 1pt] (2,2.61)-- (6,2.61);
\draw [dash pattern=on 1pt off 1pt] (2.61,2)-- (2.61,6);
\draw [dash pattern=on 1pt off 1pt] (3.62,6)-- (3.62,2);
\draw [dash pattern=on 1pt off 1pt] (4.86,2)-- (4.86,6);
\draw [dash pattern=on 1pt off 1pt] (2,4.86)-- (6,4.86);
\draw [dash pattern=on 1pt off 1pt] (6,3.62)-- (2,3.62);
\draw (6.86,2.01)-- (10.86,2.01);
\draw (10.86,2.01)-- (10.86,6.01);
\draw (10.86,6.01)-- (6.86,6.01);
\draw (6.86,6.01)-- (6.86,2.01);
\draw [dash pattern=on 1pt off 1pt] (6.86,2.63)-- (10.86,2.63);
\draw [dash pattern=on 1pt off 1pt] (7.47,2.01)-- (7.47,6.01);
\draw [dash pattern=on 1pt off 1pt] (8.47,6.01)-- (8.47,2.01);
\draw [dash pattern=on 1pt off 1pt] (9.71,2.01)-- (9.71,6.01);
\draw [dash pattern=on 1pt off 1pt] (6.86,4.87)-- (10.86,4.87);
\draw [dash pattern=on 1pt off 1pt] (10.86,3.63)-- (6.86,3.63);
\draw (16.43,2.01)-- (20.43,2.01);
\draw (20.43,2.01)-- (20.43,6.01);
\draw (20.43,6.01)-- (16.43,6.01);
\draw (16.43,6.01)-- (16.43,2.01);
\draw [dash pattern=on 1pt off 1pt] (16.43,2.63)-- (20.43,2.63);
\draw [dash pattern=on 1pt off 1pt] (17.05,2.01)-- (17.05,6.01);
\draw [dash pattern=on 1pt off 1pt] (18.05,6.01)-- (18.05,2.01);
\draw [dash pattern=on 1pt off 1pt] (19.29,2.01)-- (19.29,6.01);
\draw [dash pattern=on 1pt off 1pt] (16.43,4.87)-- (20.43,4.87);
\draw [dash pattern=on 1pt off 1pt] (20.43,3.63)-- (16.43,3.63);
\draw (11.63,1.99)-- (15.63,1.99);
\draw (15.63,1.99)-- (15.63,5.99);
\draw (15.63,5.99)-- (11.63,5.99);
\draw (11.63,5.99)-- (11.63,1.99);
\draw [dash pattern=on 1pt off 1pt] (11.63,2.61)-- (15.63,2.61);
\draw [dash pattern=on 1pt off 1pt] (12.24,1.99)-- (12.24,5.99);
\draw [dash pattern=on 1pt off 1pt] (13.24,5.99)-- (13.24,1.99);
\draw [dash pattern=on 1pt off 1pt] (14.48,1.99)-- (14.48,5.99);
\draw [dash pattern=on 1pt off 1pt] (11.63,4.85)-- (15.63,4.85);
\draw [dash pattern=on 1pt off 1pt] (15.63,3.61)-- (11.63,3.61);
\draw (6,4.2) node[anchor=north west] {$ = $};
\draw (10.9,4.2) node[anchor=north west] {$\circ$};
\draw (15.7,4.2) node[anchor=north west] {$\circ$};
\draw (2,2.61)-- (2.61,2);
\draw (2.61,3.62)-- (3.62,2.61);
\draw (3.62,4.86)-- (4.86,3.62);
\draw (4.86,6)-- (6,4.86);
\draw (6.86,3.63)-- (8.47,2.01);
\draw (8.47,4.87)-- (9.71,3.63);
\draw (9.71,6.01)-- (10.86,4.87);
\draw (11.63,3.61)-- (13.24,1.99);
\draw (13.24,5.99)-- (15.63,3.61);
\draw (16.43,2.63)-- (17.05,2.01);
\draw (17.05,3.63)-- (18.05,2.63);
\draw (18.05,6.01)-- (20.43,3.63);
\end{tikzpicture}
\caption{$\delta_a \oplus \delta_b \oplus \delta_c \oplus \delta_d=(\delta_{a+b} \oplus \delta_c \oplus \delta_d) \c
        (\delta_{a+b} \oplus \delta_{c+d}) \c
        (\delta_a \oplus \delta_b \oplus \delta_{c+d})
$}
\label{figure_l4}
\end{figure}
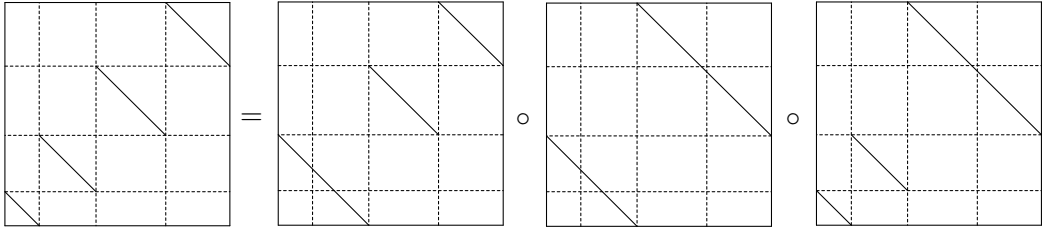

This theorem raises the question whether $\La_k$ could be 2-composable for $k \geq 4$. Our work
from Section \ref{ch2} quickly determines that this is not the case.
\begin{prop}
$\La_k$ for $k\geq 4$ is not 2-composable. In particular, it is not $n$-composable for
any even number $n$.
\end{prop}
\begin{proof}
Since $\La_k$ is an infinite class which does not contain $\I$ the statement  directly follows from Theorem \ref{evencomposable}.
\end{proof}

We have now covered the classes $\La_k$ for all $k \geq 2$.
It remains to consider the class~$\La$, which we show to be uncomposable. Before we proceed with the 
proof, we introduce an additional useful concept.
We call a~subsequence $s$ of a~permutation $\pi$ a~\emph{block} if $s$ is either an increasing
or a~decreasing contiguous subsequence of consecutive integers.
We then call $\pi$ a~$k$\emph{-block} if it is a~concatenation of at most $k$ blocks (see Figure \ref{figure_blocks}).

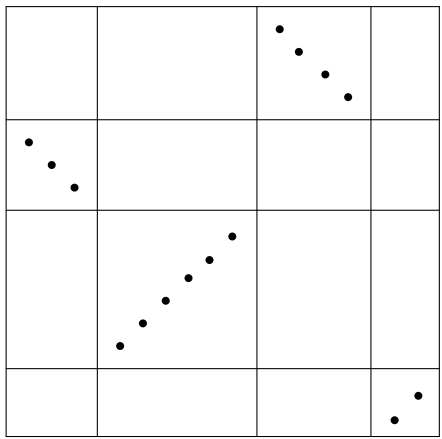
\begin{figure}[ht]
\centering
\begin{tikzpicture}[line cap=round,line join=round,>=triangle 45,x=0.6cm,y=0.6cm]
\clip(0.25,-6.39) rectangle (10.17,3.7);
\draw (0.5,3.5)-- (0.5,-6);
\draw (0.5,-6)-- (10,-6);
\draw (10,-6)-- (10,3.5);
\draw (10,3.5)-- (0.5,3.5);
\draw (0.5,1)-- (10,1);
\draw (0.5,-1)-- (10,-1);
\draw (0.5,-4.5)-- (10,-4.5);
\draw (2.5,3.5)-- (2.5,-6);
\draw (6,-6)-- (6,3.5);
\draw (8.5,3.5)-- (8.5,-6);
\begin{scriptsize}
\fill [color=black] (1,0.5) circle (1.5pt);
\fill [color=black] (1.5,0) circle (1.5pt);
\fill [color=black] (2,-0.5) circle (1.5pt);
\fill [color=black] (3,-4) circle (1.5pt);
\fill [color=black] (3.5,-3.5) circle (1.5pt);
\fill [color=black] (4,-3) circle (1.5pt);
\fill [color=black] (4.5,-2.5) circle (1.5pt);
\fill [color=black] (4.96,-2.1) circle (1.5pt);
\fill [color=black] (5.46,-1.58) circle (1.5pt);
\fill [color=black] (6.5,3) circle (1.5pt);
\fill [color=black] (6.92,2.5) circle (1.5pt);
\fill [color=black] (7.5,2) circle (1.5pt);
\fill [color=black] (8,1.5) circle (1.5pt);
\fill [color=black] (9.02,-5.64) circle (1.5pt);
\fill [color=black] (9.54,-5.1) circle (1.5pt);
\end{scriptsize}
\end{tikzpicture}
\caption{An example of a 4-block}
\label{figure_blocks}
\end{figure}

\begin{lemma} \label{lemma_blocks}
Let $\pi \in S_n$ be a~$k$-block and let $\sigma \in S_n$ be an $l$-block. Then $\pi \c \sigma$ is a~$(k\cdot l)$-block.
\end{lemma}
\begin{proof}
Choose a~block of $\sigma$ at indices $a, a+1, \ldots, a+b$. Then the sequence 
$$\pi(\sigma(a)), \pi(\sigma(a+1)), \ldots, \pi(\sigma(a+b))$$
is a~contiguous subsequence of either 
$\pi(1), \pi(2), \ldots, \pi(n)$ or $\pi(n), \ldots, \pi(2), \pi(1)$
and therefore is a~concatenation of at most $k$ blocks  since $\pi$ itself is a~$k$-block. 
This is true for each of the $l$ blocks of $\sigma$,
therefore $\pi \c \sigma$ is a~$(k\cdot l)$-block.
\end{proof}

Now we can prove our main result.

\begin{thm}\label{l}
The class $\La$ is not composable.
\end{thm}
\begin{proof}

Every subclass of $\La$ is determined by one or more forbidden layered permutations.
If $\La$ is composable from $k$ subclasses, we may choose one forbidden layered permutation from each of them
and then choose $n$ large enough so that $\pi = \bigoplus_{i=1}^{n+1}\delta_{n+1}$ contains all of the chosen patterns.
That way, $\La \subseteq \C^k$ where $\C = \Av(\pi) \cap \La$.

Clearly every permutation in $\C$ has at most $n$ layers longer than $n$, otherwise it would contain $\pi$.
Our goal is to show that permutations in $\C^k$ are somehow very close to patterns composed from permutations 
that have a~constant number of
non-trivial layers and all other layers are just of size 1. Given a layered permutation we call
a layer of length at most $n$ a \emph{short} layer and a layer of length more than $n$ a \emph{long} layer. 

We say that two permutations $\alpha$ and $\beta$ are $(c,l)$\emph{-close}, if $|\alpha(i) - \beta(i)| \leq c$
for every index $i$ with at most $l$ exceptions.


For $\sigma \in \C$ we denote by $N(\sigma)$ the permutation created from $\sigma$ by replacing every short layer
by the corresponding number of layers of size 1, i.e. flipping the short layers into increasing blocks. 

We can now formally state our goal: we shall prove that for any $\sigma_1, \sigma_2, \ldots, \sigma_k \in \C$
the permutations $\sigma_k \c \cdots \c \sigma_2 \c \sigma_1$ and $N(\sigma_k) \c \cdots \c N(\sigma_2) \c N(\sigma_1)$
are $(2nk, 8n^2k^2)$-close. We will prove this by induction on $k$.

If $k=1$, we have to show that $\sigma_1$ and $N(\sigma_1)$ are $(2n,8n^2)$-close. Since $N(\sigma)$ is created
by manipulating layers of $\sigma$ of length at most $n$ in place, every element of $\sigma$ is shifted by at most $n$, so they
are even $(n,0)$-close, thus the first step of induction is done.

If $k \geq 2$, suppose that $\sigma = \sigma_k \c \cdots \c \sigma_2 \c \sigma_1$ and 
$\nu = N(\sigma_k) \c \cdots \c N(\sigma_2) \c N(\sigma_1)$
are $(2nk, 8n^2k)$-close and we shall prove the statement for $k+1$.

Given a~layered permutation and one of its layers of size $l+1$ at indices $i, i+1, \ldots, i+l$,
we say that a~number $u$ is \emph{in the area of influence} of this layer if $i \leq u \leq i+l$.

Given a long layer of $\sigma_{k+1}$ or $N(\sigma_{k+1})$ (their long layers are the same),
there are at most $4nk$ indices $u$ such that $|\sigma(u) - \nu(u)| \leq kn$ and
 $\nu(u)$ is in the area of influence of this layer and $\sigma(u)$ is not: at most $2nk$ to the left and to the right of the layer.
 Similarly there are at most $4nk$ indices $u$ such that $|\sigma(u) - \nu(u)| \leq kn$ and $\sigma(u)$ is in the area
 of influence of the considered layer and $\nu(u)$ is not. Since there are at most $n$ long layers, we get that
 in total there are at most $8n^2k$ indices $u$ such that $|\sigma(u) - \nu(u)| \leq kn$ and one of $\{\sigma(u), \nu(u)\}$
 is in the area of influence of a long layer while the other is not in that area.

 By the induction hypothesis, there are at most $8n^2k^2$ indices $u$ such that $|\sigma(u) - \nu(u)| > nk$.
 Together with the at most $8n^2k$ indices from the previous paragraph we get $8n^2k^2 + 8n^2k = 8n^2k(k+1) \leq 8n^2(k+1)^2$
 indices at which we will allow $\sigma_{k+1} \c \sigma$ and $N(\sigma_{k+1}) \c \nu$ to differ arbitrarily in our
 proof that these two permutations are $(2n(k+1), 8n^2(k+1)^2)$-close. It remains to show 
 $|\sigma_{k+1}(\sigma(u)) - N(\sigma_{k+1}(\nu(u))| \leq n(k+1)$ for all the remaining indices $u$ to complete the induction step.


For other indices $u \in \{1, 2, \ldots, |\sigma|\}$ not considered so far it holds that $|\nu(u) - \sigma(u)| \leq 2nk$
and that either $\nu(u)$ and $\sigma(u)$ are both in the area of influence of the same long layer of $\sigma_{k+1}$
or they are in areas of influence of short or trivial layers. In the latter case the value of $\sigma(u)$ changes by
at most $n$ after applying $\sigma_{k+1}$ to it and similarly the value of $\nu(u)$ changes by at most $n$ after applying
$N(\sigma_{k+1})$ to it, thus
$|N(\sigma_{k+1})(\nu(u)) - \sigma_{k+1}(\sigma(u))| \leq 2nk + 2n \leq n(k+1)$. In the former case
it is enough to realise that for a given decreasing permutation $\delta_a$ it holds that
$\delta_a(x\pm y) = \delta_a(x) \mp y$, thus if $\sigma(u)$ differs by
$y$ from $\nu(u)$ and they are in the area of influence of the same long layer, after applying $\sigma_{k+1}$ (or $N(\sigma_{k+1})$, which
is the same for the big layers) the values still differ by $y \leq 2nk \leq 2n(k+1)$, which finishes
the induction step.

Notice that for $\sigma \in \C$ the permutation $N(\sigma)$ is a~$(2n)$-block
according to the definition above. 
Thus by Lemma \ref{lemma_blocks}
we get that by composing $k$ such permutations we get a~permutation which is a~$(2n)^k$-block.
As a~result we get that each permutation from $(\C)^k$ is $(c,l)$-close 
to a~$C$-block for suitable fixed constants $c,l,C$. Notice
now that every $C$-block avoids the $(C+1)$-block $\gamma = 214365\cdots(2C+2)(2C+1)$,
so every permutation from $(\C)^k$ is $(c,l)$-close to a~permutation avoiding $\gamma$. We
can construct a~layered permutation which is not $(c,l)$-close to any $\gamma$-avoider as follows.
Choose a~layered permutation with $C+1$ layers of length $l+2c+1$ and consider a~permutation
$(c,l)$-close to it. Then in every layer there are at least $2c+1$ elements whose value 
changed by at most $c$; therefore there exist at least two elements which remained in decreasing order.
Choosing these two elements from every layer forms an occurrence of $\gamma$.
Since $\La$ contains a~permutation which is not $(c,l)$-close to $\gamma$ and $(\C)^k$
does not contain such permutations, we get that $\La \nsubseteq (\C)^k$, achieving contradiction.
\end{proof}

Preceding results and Lemma \ref{lemma_symcomp} imply the following corollary.
\begin{cor}
The classes of co-layered permutations $\La_2^c$, $\La_3^c$ and $\La^c$ are not composable.
\end{cor}

\section{Other results}
\label{ch5}

In the final section of this work we collect several miscellaneous results concerning composability.
First we provide more examples of composable classes, and then we fininish by presenting
 several additional examples of uncomposable classes.

\subsection{Composable principal classes}

In this section, we use results of Section \ref{ch3} and of \cite{JelinekValtr13}
to prove that many classes avoiding a~single decomposable pattern (a permutation which can
be written as a~non-trivial sum of smaller permutations) are composable.

We will base our proof on the following splittability result of Jelínek and Valtr \cite{JelinekValtr13}.

\begin{lemma}[Jelínek, Valtr \cite{JelinekValtr13}]\label{splitlemma}
Let $\alpha, \beta, \gamma$ be three nonempty permutations and let $\pi \in \Av(\alpha \oplus \beta \oplus \gamma)$.
Then $\pi$ can be merged from two sequences $(a)_{i=1}^n$ and $(c)_{i=1}^m$ such that $a$ avoids $\alpha \oplus \beta$,
$c$ avoids $\beta \oplus \gamma$ and for any $i \in [n]$ and $j \in [m]$ either $\pi^{-1}(a_i) < \pi^{-1}(c_j)$ or $a_i < c_j$.
\end{lemma}

\begin{thm}\label{thm_52}
  If $\alpha$ and $\gamma$ are any non-empty permutations and $\beta = \delta_n$ for a~positive integer $n$, then
  $$\Av(\alpha \oplus \beta \oplus \gamma) \subseteq (\V(\Av(\alpha \oplus \beta),\Av(\beta\oplus \gamma))
\cap \Av(\alpha \oplus \beta \oplus \gamma)) \c \Ho.$$
In particular, $\Av(\alpha \oplus \delta_n \oplus \gamma)$ is 2-composable whenever $\alpha \oplus \delta_n \oplus \gamma
\notin \Ho$.
\end{thm}
\begin{proof}

Let $\C = \Av(\alpha \oplus \beta \oplus \gamma)$, $\A = \Av(\alpha \oplus \beta)$ and $\B = \Av(\beta \oplus \gamma)$.
Lemma \ref{splitlemma} and Corollary \ref{cor_splitcompose} immediately imply that
$\C \subseteq \V(\A,\B) \c \Ho.$

Let $\pi \in \C$ be merged from sequences $a$ and $c$ as in Lemma \ref{splitlemma}
and let $\sigma = ac$. We have to show that $\sigma \in \C$. Suppose for a~contradiction that $\sigma$
contains a~copy of $\alpha \oplus \beta \oplus \gamma$. Let $b$ be the decreasing subsequence of $\sigma$
representing the occurrence of $\beta$. Then $b$ cannot be
contained entirely in $a$ or in $c$ since that would create a~copy of $\alpha \oplus \beta$ in $a$
or of $\beta \oplus \gamma$ in $c$. Thus if $\beta = 1$ the contradiction is reached immediately. 

If $|\beta| > 1$, we would like to show that $b$ is also a~subsequence of $\pi$. Assume it is not,
therefore there are elements $b_i$ and $b_j$ with $i < j$ such that they appear in reverse
order in $\pi$. That can only be achieved if $b_i$ is in $a$ and $b_j$ is in $c$,
which together with $b_i > b_j$ contradicts the properties of $a$ and $c$ from Lemma \ref{splitlemma}.

It follows that the entire occurrence of $\alpha \oplus \beta \oplus \gamma$ is also contained in $\pi$,
which is a~contradiction, thus $\C \subseteq (\V(\A, \B) \cap \C) \c \Ho$.

To prove that $\C$ is really 2-composable for $\alpha \oplus \beta \oplus \gamma \notin \Ho$ it remains
to verify that $\V(\A,\B) \cap \C$ and $\Ho$ are proper subclasses of $\C$. Clearly
$\V(\A,\B) \cap \C \subseteq \C$ and the condition $\alpha \oplus \beta \oplus \gamma \notin \Ho$ implies
$\Ho \subseteq \C$, so it remains to show that the inclusions are proper.
Consider the permutation $(\alpha \oplus \beta) \ominus (\beta \oplus \gamma)$ which is clearly in $\C$
and not in $\V(\A,\B)$. For the class $\Ho$ we use the results of Atkinson, who showed in \cite[Proposition 3.4]{Atkinson99} 
that the class $\Ho$ has
a basis of size 3~and therefore it cannot be equal to a~principal class.
\end{proof}

Note that for the case $\beta = 1$ we get 
\begin{equation*} \V(\Av(\alpha \oplus 1), \Av(1 \oplus \gamma)) \subsetneq \Av(\alpha \oplus 1 \oplus \gamma), \end{equation*}
and thus we may omit the intersection with $\Av(\alpha \oplus 1 \oplus \gamma)$ in the formula of Theorem \ref{thm_52}.
Indeed, if a~permutation is concatenated of two parts, first avoiding $\alpha \oplus 1$ and the second avoiding $1 \oplus \gamma$,
such a~permutation cannot contain an occurrence of $\alpha \oplus 1 \oplus \gamma$ since one of the two parts would contain
the middle 1~and thus the forbidden pattern.

\subsection{More uncomposable classes}

So far we have used classes such as $\V$ or $\Ho$ to prove that other classes are composable.
In this section, we will show that these classes, and classes similar to them, are themselves uncomposable.

We call a~permutation $\eta \in \Ho$ \emph{alternating} if $\eta(2i-1) < \eta(2i) > \eta(2i+1)$ for all
possible values of $i$. We will use the following simple observation about alternating permutations in $\Ho$.
\begin{figure}[ht]
\centering
\begin{tikzpicture}[line cap=round,line join=round,>=triangle 45,x=0.5cm,y=0.5cm]
\clip(-2.5,-3.09) rectangle (5.52,4.99);
\draw (-2,4.5)-- (-2,-2.5);
\draw (-2,-2.5)-- (5,-2.5);
\draw (5,-2.5)-- (5,4.5);
\draw (5,4.5)-- (-2,4.5);
\draw (-2,1.5)-- (5,1.5);
\begin{scriptsize}
\fill [color=black] (-1.5,-2) circle (1.5pt);
\fill [color=black] (-0.5,2) circle (1.5pt);
\fill [color=black] (0.5,-1) circle (1.5pt);
\fill [color=black] (1.5,3) circle (1.5pt);
\fill [color=black] (2.5,0) circle (1.5pt);
\fill [color=black] (3.5,4) circle (1.5pt);
\fill [color=black] (4.5,1) circle (1.5pt);
\end{scriptsize}
\end{tikzpicture}
\caption{The alternating permutation of length 7}
\end{figure}
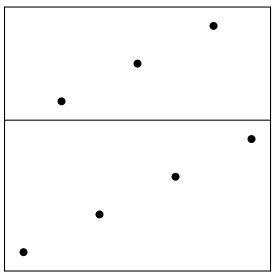

\begin{obs}\label{obs_alt}
Every permutation from $\Ho$ is contained in an alternating permutation from $\Ho$.
\end{obs}

\begin{prop} \label{VHstuff}
The classes $\V$, $\V^c$, $\V(\D, \I)$, $\V(\I, \D)$, $\Ho$, $\Ho^c$, $\Ho(\I, \D)$ \\
and $\Ho(\D, \I)$ are not composable.
\end{prop}
\begin{proof}
We will show the proof for the class $\Ho$, the same approach can be applied to every mentioned 
horizontal merge and the result is transferred by inversion to the vertical merges by to Corollary \ref{cor_inversion}.

Suppose that $\Ho$ is composable from its proper subclasses $\C_1, \ldots, \C_k$. Each of $\C_1, \ldots, \C_k$
avoids a permutation of $\Ho$, thus according to Observation \ref{obs_alt} there is an alternating
permutation $\eta \in \Ho$ such that $\C_i \subseteq \Av(\eta)$ and therefore if $\C = \Av(\eta) \cap \Ho \subsetneq \Ho$
we have $\Ho \subseteq (\C)^k$.

Any permutation $\pi \in \C$ is merged from two sequences $a$ and $b$ of consecutive integers.
We label elements of $\pi$ by $a$ or $b$ depending on which sequence they belong to.
A sequence of elements with alternating labels forms a~copy of an alternating permutation in $\pi$.
The length of the longest sequence of alternating labels in $\pi$ is thus limited by a~constant $N$ determined
by the order of $\eta$, thus $\pi$ can be broken into at most $N$ contiguous parts each having one label.
Since elements labeled with a~single label form a~sequence of consecutive integers, this
implies that $\pi$ is in fact an $N$-block. Since the choice of $\pi$ was arbitrary,
every permutation of $\C$ is an $N$-block and by Lemma \ref{lemma_blocks} every permutation of $(\C)^k$
is an $(N^k)$-block. But a~long enough alternating permutation from $\Ho$ is not an $(N^k)$-block,
therefore $\Ho \nsubseteq (\C)^k$ and the proof is finished.
\end{proof}

\section{Conclusion}

This paper studies the previously unexplored concept of composability of permutation classes.
Given a~permutation class, our main goal is to show, how it can be constructed using smaller permutation classes
and the composition operator, or to prove that this cannot be done. Throughout the paper,
we present both types of results. 

On the positive side, Theorems \ref{thm_Ik} and \ref{thm_k+l-1}
show two distinct ways of constructing the class $\Av(k\cdots 21)$, 
Theorem \ref{l4} provides infinitely many examples of
classes of layered patterns which can be constructed from simpler subclasses
and Theorem \ref{thm_52} shows that many principal classes
avoiding a~decomposable pattern are composable.

On the negative side, in Theorems \ref{l2}, \ref{f2}, \ref{l3} and \ref{l}
we present four different classes of layered patterns which cannot be constructed from any number
of proper subclasses using composition, and Proposition \ref{VHstuff} provides us with
8 more examples of uncomposable classes.

Composability is similar to splittability in that both these properties describe
how a~bigger class is built from smaller ones. We do not know whether these
two properties are somehow connected; however, our research suggests that this may be the case,
since every composable class we have found so far is also splittable. We have
found examples of splittable yet uncomposable classes, namely the classes $\La_2$ and $\La_3$
introduced in Section \ref{ch4}. The class of all layered permutations is an example of a~both uncomposable
and unsplittable class. The last case remains open and we pose it as a~question for future work.

\begin{question}
Is there a~permutation class which is composable and unsplittable?
\end{question}

In Section 3 we showed that if a class is composable and avoids an increasing pattern, then its reverse
and complement are composable. It remains open whether the avoidance condition is necessary.

\begin{question}
Is there a composable class $\A$ such that $\A^r$ or $\A^c$ is not composable?
\end{question}

Splittability has the property that if a class is splittable, it can be split into two parts. This is not the case
for composability as we showed in
Section 4 where we proved that the class $\La_k$ for $k \geq 4$ is $3$-composable but not 2-composable.

\begin{question}
Is there a $4$-composable class which is not $3$-composable? More generally, 
is there a universal constant $K$ such that every composable class is $K$-composable?
\end{question}

Our work may find applications in enumerating permutation classes. Denote by $\gr(\A)$ the growth
rate of the class $\A$ as defined e.g. in \cite{Vatter15}.
It is not difficult to see that if $\C \subseteq \A \c \B$, then $\gr(\C) \leq \gr(\A) \cdot \gr(\B)$.
Using this observation one could try to find upper bounds for growth rates of composable permutation classes.

\section{Acknowledgement}
I am very grateful to Vít Jelínek for many consultations, valuable advice and proofreading of this work.

\bibliographystyle{abbrv}

\renewcommand{\bibname}{References}


\bibliography{bibliography}


%
%

\end{document}